\documentclass[10pt,a4paper]{article}
\usepackage{amsmath,amsfonts,amsthm,comment,url}
\usepackage{graphicx}
\usepackage{hyperref}

\theoremstyle{plain}
\newtheorem{theorem}{Theorem}

\theoremstyle{definition}

\newcommand{\EE}{\mathbb{E}}
\newcommand{\Bin}{\operatorname{Bin}}

\begin{document}

\title{The law of thin processes: \\a law of large numbers for point processes}
\author{Matthew Aldridge}
\date{February 2025}

\maketitle

\begin{abstract}
If you take a superposition of $n$ IID copies of a point process and thin that by a factor of $1/n$, then the resulting process tends to a Poisson point process as $n\to\infty$. We give a simple proof of this result that highlights its similarity to the law of large numbers and to the law of thin numbers of Harremoës et~al.
\end{abstract}

\section{Theorems}

The law of large numbers is the following result.

\begin{theorem}[Law of large numbers] \label{largenum}
Let $X$ be a real-valued random variable with finite expectation $\mu = \EE X$. Let $X_1, X_2, \dots$ be IID copies of $X$. Then the scaled sum
\[ \frac{1}{n}(X_1 + X_2 + \cdots + X_n) \]
tends in distribution to a point mass at $\mu$ as $n \to \infty$.
\end{theorem}

(The weak law of large numbers is more often presented as concerning convergence in probability to $\mu$, but convergence in probability to a constant and convergence in distribution to a point mass at that constant are equivalent.)

The steps here are: first, we sum $n$ IID copies of a random variable; second, we scale in by $1/n$; then we converge to a distribution that depends only on the first moment $\mu$ of the random variable we started with.

Harremoës, Johnson and Kontoyiannis \cite{HJK} give a similar result for discrete random variables that take values in the non-negative integers, where the limit theorem remains a statement about the non-negative integers. While the law of large numbers would still apply for such a random variable, the act of scaling the sum by $1/n$ takes us outside the non-negative integers, and the limiting distribution of a point mass at $\mu$ is also no longer supported on the non-negative integers (unless $\mu$ itself  happens to be an integer). Harremoës et~al make two changes to the law of large numbers to remedy this. The first change is that the scaling operation is replaced with with \emph{thinning}. To get the thinning $p \circ X$ of a non-negative integer random variable $X$, one thinks of $X$ as representing a number of items, each of which is independently kept with probability $p$ and removed with probability $1 - p$; more formally, the conditional distribution of ${p \circ X}$ given $X$ is the binomial distribution $\Bin(X, p)$. The second change is that the limiting distribution of the point mass at $\mu$ is replaced by a Poisson distribution with rate $\mu$. Harremoës et~al call this result `the law of thin numbers'.

\begin{theorem}[Law of thin numbers \cite{HJK}] \label{thinnum}
Let $X$ be a random variable on the non-negative integers with finite expectation $\mu = \EE X$. Let $X_1, X_2, \dots$ be IID copies of $X$. Then the thinned sum
\[ \frac{1}{n}\circ (X_1 + X_2 + \cdots + X_n) \]
tends in distribution to a Poisson distribution with rate $\mu$ as $n \to \infty$.
\end{theorem}

In this note, we show that there exists a version of this for point processes too. Consider a point process $\xi$. First, take a superposition of $n$ IID copies of $\xi$. We think of a point process as a random counting measure, so the superposition is a sum of $n$ IID copies of the random measure. Second, thin that superposition by a factor of $1/n$. To get the thinning $p \circ \xi$ of a point process $\xi$, we independently keep each point of $\xi$ with probability $p$ and delete it with probability $1 - p$. 
Then, as $n \to \infty$, we get a Poisson process whose intensity measure is that of the original point process. This result is the natural equivalent for point processes of the laws of large and thin numbers. 
Following Harremoës et al, we propose the name `the law of thin processes'.

We largely use notation and terminology from the book \cite{devere}.

\begin{theorem}[Law of thin processes]\label{thinproc}
Let $\xi$ be a point process on a complete separable metric space $\mathcal X$ with intensity measure $\mu$, where $\mu(A) = \operatorname{\mathbb E}\xi(A)$. Let $\xi_1, \xi_2, \dots$ be IID copies of $\xi$. Then the thinned superposition
\[ \frac{1}{n}\circ (\xi_1 + \xi_2 + \cdots + \xi_n) \]
tends weakly to a Poisson process on $\mathcal X$ with intensity measure $\mu$ as $n \to \infty$.
\end{theorem}

Although I have not managed to find the law of thin processes exactly as written above  in the literature, there are many results on the same lines. The limit theorem in \cite{serforzo} for more general thinning procedures is very similar, and the law of thin processes is essentially a special case of that result. There are two results in \cite{devere} for point processes in $\mathbb R$ that replace one of the two operations (superposition or thinning) with dilation or contraction of space:
\begin{itemize}
    \item If one first takes the superposition of $n$ IID copies of a real-valued point process, and, second, dilates space by a factor of $n$ (instead of thinning with probability $1/n$), then the limit is a homogenous Poisson process. This further requires the point process to be stationary. \cite[Proposition 11.2.VI]{devere}.
    \item If one takes a a real-valued point process and, first, contracts space by a factor of $1/n$ (instead of taking the superposition of $n$ IID copies), and, second, thins that with probability $1/n$, then the limit is a Poisson process. This further requires an extra condition which is a sort of weak stationarity condition. \cite[Proposition 11.3.I]{devere}
\end{itemize}
Moreover, the law of thin processes can presumably be shown as a corollary of much more general (and complicated) Poisson limit theorems, such as \cite[Propositions 11.2.V or 11.3.III]{devere}.
However, our emphasis in this note is not on novelty of the result itself, nor on presenting Poisson convergence in greatest generality, but rather on giving a simple proof of Theorem \ref{thinproc} that exactly reflects standard proofs for the laws of large and thin numbers.

\section{Proofs}

\subsection{Laws of large and thin numbers}

We start by sketching the proofs of the laws of large and thin numbers, because their structures -- which are extremely similar -- will show us the route to take to prove the law of thin processes and will highlight the similarities between the three results.

The proof of the law of large numbers will use the Laplace transform $L_X(u) = \operatorname{\mathbb E}\mathrm{e}^{-uX}$ for $u$ a real number in some neighbourhood $U$ of $0$. (One could just as well use the moment generating function $M_X(u) = L_X(-u) = \operatorname{\mathbb E}\mathrm{e}^{uX}$, but we keep the minus sign here for consistency with later proofs. For the purpose of this proof we'll assume that $L_X(u)$ does exist within such a $U$, but if convergence is an issue, one can instead use the characteristic function $\Phi_X(u) = L_X(-\mathrm{i}u) = \operatorname{\mathbb E}\mathrm{e}^{\mathrm{i}uX}$, which exists for all $u \in \mathbb R$.)

The Laplace transform has the following three crucial properties:
\begin{enumerate}
\item Behaviour under independent sums: If $X$ and $Y$ are independent, then $L_{X+Y}(u) = L_X(u)\,L_Y(u)$ for all $u \in U$.
\item Behaviour under scaling: $L_{aX}(u) = L_X(au)$ for all $a \in \mathbb R$ and $u \in U$.
\item Behaviour with limits: If $L_{X_n}(u) \to L_X(u)$ for all $u \in U$ as $n \to \infty$, then $X_n \to X$ in distribution.
\end{enumerate}


\begin{proof}[Proof of Theorem \ref{largenum}]
The Laplace transform of $X$ is
\[ L_X(u) = \operatorname{\mathbb E}\mathrm{e}^{-uX} = \operatorname{\mathbb E} \sum_{j=0}^\infty \frac{(-uX)^j}{j!} = \sum_{j=0}^\infty \frac{(-1)^j}{j!} \,(\mathbb EX^j)\,u^j = 1 - \mu u + o(u). \]
Writing
\[ Y_n = \frac{1}{n}(X_1 + X_2 + \cdots + X_n) \]
for the scaled sum, we have, using facts 1 and 2 above,
\[ L_{Y_n}(u) = L_X \Big( \frac1n\,u\Big)^n = \bigg(1 - \frac{\mu u}{n} + o\Big(\frac1n\Big)\bigg)^n \to \mathrm{e}^{-\mu u} \]
as $n \to \infty$. But this is the Laplace transform of a point mass at $\mu$, so by fact~3, $Y_n$ tends to that point mass as $n \to \infty$.
\end{proof}

The proof of the law of thin numbers will use the alternate probability generating function $A_X(u) = \operatorname{\mathbb E}(1-u)^X$ for $u \in [0,2]$. The alternate probability generating function has the following three crucial properties, equivalent to the three properties of the Laplace transform:
\begin{enumerate}
\item Behaviour under independent sums: If $X$ and $Y$ are independent, then $A_{X+Y}(u) = A_X(u)\,A_Y(u)$ for all $u \in [0,2]$.
\item Behaviour under thinning: $A_{p\circ X}(u) = A_X(pu)$ for all $p \in [0,1]$ and $u \in [0,2]$.
\item Behaviour with limits: If $A_{X_n}(u) \to A_X(u)$ for all $u \in [0,2]$ as $n \to \infty$, then $X_n \to X$ in distribution.
\end{enumerate}
Fact 2 is why we prefer the alternate probability generating function to the more common probability generating function $G_X(u) = \operatorname{\mathbb E}u^X = A_X(1 - u)$, since that has the more awkward expression $G_{p \circ X}(u) = G_X(1 - p + pu)$.

The following proof of the law of thin numbers is essentially that given in \cite[Proposition 2.1]{JK}.

\begin{proof}[Proof of Theorem \ref{thinnum}]
The alternate probability generating function of $X$ is
\[ A_X(u) = \operatorname{\mathbb E}(1-u)^X = \operatorname{\mathbb E} \sum_{j=0}^\infty \binom{X}{j}(-u)^j = \sum_{j=0}^\infty \frac{(-1)^j}{j!} \,\mathbb E(X)_j\,u^j = 1 - \mu u + o(u) ,\]
where
\[ \mathbb E(X)_j = \mathbb EX(X-1)\cdots(X-j+1) \]
are the factorial moments. Writing
\[ Y_n = \frac{1}{n}\circ(X_1 + X_2 + \cdots + X_n) \]
for the thinned sum, we have, using facts 1 and 2 above,
\[ A_{Y_n}(u) = A_X \Big( \frac1n\,u\Big)^n = \bigg(1 - \frac{\mu u}{n} + o\Big(\frac1n\Big)\bigg)^n \to \mathrm{e}^{-\mu u} \]
as $n \to \infty$. But this is the alternate probability generating function of a Poisson distribution with rate $\mu$, so by fact 3, $Y_n$ tends to that Poisson distribution as $n \to \infty$.
\end{proof}

\subsection{Law of thin processes}

To prove the law of thin processes (Theorem \ref{thinnum}), we use what we shall call the \emph{alternate probability generating functional}
\[ A_\xi(u) = \operatorname{\EE} \exp \left( \int \log\big(1 - u(x)\big)\,\xi(\mathrm dx)\right)  \]
for $u \in \mathcal U$, where $\mathcal U = \mathcal U(\mathcal X)$ is the set of functions $u\colon\mathcal X \to [0,1]$ that are zero outside a bounded set. Because the point process is almost surely finite on the set where $u$ does not vanish, we can write
\[ A_\xi(u) = \mathbb E \prod_i \big(1 - u(x_i)\big) , \]
where the product is taken over the points $x_i$ of $\xi$.

The alternate probability generating functional again has three crucial properties, which mirror those we saw for the Laplace transform and the alternative probability generating function in the two earlier proofs:
\begin{enumerate}
\item Behaviour under independent superpositions: If $\xi$ and $\eta$ are independent, then $A_{\xi+\eta}(u) = A_\xi(u)\,A_\eta(u)$ for all $u \in \mathcal U$. \cite[Proposition 9.4.IX]{devere}
\item Behaviour under thinning: $A_{p \circ \xi}(u) = A_\xi(pu)$ for all $p \in [0,1]$ and $u \in \mathcal U$. \cite[equation (11.3.2)]{devere}
\item Behaviour with limits: If $A_{\xi_n}(u) \to A_\xi(u)$ for all $u \in \mathcal U$ as $n \to \infty$, then $\xi_n \to \xi$ weakly. \cite[Proposition 11.1.VIII]{devere}
\end{enumerate}
Again, fact 2 is why we prefer the alternate probability generating functional to the more common probability generating functional 
\[ G_\xi(u) = \operatorname{\EE} \exp \left( \int \log u(x)\,\xi(\mathrm dx)\right) = \mathbb E \prod_i u(x_i) = A_\xi(1-u), \]
which has the more awkward expression $G_{p \circ \xi}(u) = G_\xi(1 - p + pu)$, or the Laplace functional
\[ L_\xi(u) = \operatorname{\EE} \exp \left( -\int u(x)\,\xi(\mathrm dx)\right) = \mathbb E \prod_i \mathrm{e}^{-u(x_i)} = A_\xi(1-\mathrm{e}^{-u}), \]
for which $L_{p \circ \xi}(u) = L_\xi(-\log(1 - p + p\mathrm{e}^{-u}))$.

The final preparatory step we need is a result that writes the alternate probability generating functional in terms of the factorial moment measures $m_{(j)}$, in the same way as we earlier wrote the Laplace transform in term of the moments $\mathbb EX^j$ and the alternate probability generating function in terms of the factorial moments $\mathbb E(X)_j$. Informally, we would expect 
\[
A_\xi(u) = 1 + \sum_{j=1}^\infty \frac{(-1)^j}{j!} \int_{\mathcal X^{j}} u(x_1)\cdots u(x_j)\,m_{(j)}(\mathrm dx_1 \times \cdots \times \mathrm dx_j) .
\]
However, there are convergence issues that mean the summands on the right-hand side may not exist or the sum may not converge (see \cite[Lemma 4.11]{penrose}). What we do have if the following: if $\xi$ is such that the $k$th factorial moment measure $m_{(k)}$ exists, then
\[
A_\xi(pu) = 1 + \sum_{j=1}^k \frac{(-p)^j}{j!} \int_{\mathcal X^{j}} u(x_1)\cdots u(x_j)\,m_{(j)}(\mathrm dx_1 \times \cdots \times \mathrm dx_j) + o(p^k)
\]
for $u \in \mathcal U$ as $p \to 0$ \cite[Proposition 9.5.VI]{devere}. The exact definition of the factorial moment measures is not important here, as we will only need the $k = 1$ case,
\begin{equation} \label{want}
A_\xi(pu) = 1 - p\int_\mathcal X u(x)\,\mu(\mathrm{d}x) + o(p) , 
\end{equation}
where the first factorial moment measure $m_{(1)} = \mu$ is the simply the intensity measure $\mu(A) = \operatorname{\mathbb E}\xi(A)$ of $\xi$.

We are now ready to prove the law of thin processes by following exactly the steps of the two earlier proofs.

\begin{proof}[Proof of Theorem \ref{thinproc}]
Writing
\[ \eta_n = \frac{1}{n}\circ(\xi_1 + \xi_2 + \cdots + \xi_n) \]
for the thinned superposition, we have, using facts 1 and 2 above,
\[ A_{\eta_n}(u) = A_\xi \Big( \frac1n\,u\Big)^n = \bigg(1 - \frac{1}{n} \int u(x)\,\mu(\mathrm{d}x)+ o\Big(\frac1n\Big)\bigg)^n  \]
as $n \to \infty$, where the second equality is the result \eqref{want}. Then
\[ A_{\eta_n}(u) = \bigg(1 - \frac{1}{n} \int u(x)\,\mu(\mathrm{d}x)+ o\Big(\frac1n\Big)\bigg)^n \to \exp \bigg( -\int u(x)\,\mu(\mathrm{d}x) \bigg) . \]
But this is the alternate probability generating functional of a Poisson process with intensity measure $\mu$ \cite[Example 9.4(c)]{devere}, so by fact 3, $\eta_n$ tends to that Poisson process as $n \to \infty$.
\end{proof}


\end{document}